\documentclass[preprint,12pt]{elsarticle}

\usepackage{graphicx}
 \usepackage{epsfig}
\usepackage{amssymb}
\usepackage{amsmath}
\usepackage{dsfont}
\usepackage{pb-diagram}
\usepackage{bbm}

\usepackage{youngtab}

\newcommand{\HH}{\mathcal{H}}
\newcommand{\EE}{\mathcal{E}}
\newcommand{\DD}{\mathcal{D}}
\newcommand{\Sym}{\mathfrak{S}}
\newcommand{\double}{\mathbbmtt}
\newlength{\tangle}
\newcommand{\htangle}[2]{\settoheight{\tangle}
  {\epsfysize=#2 \epsfbox{figures/tangle.#1}}
  \raisebox{-.5\tangle}{\epsfysize=#2 \epsfbox{figures/tangle.#1}}}
\newcommand{\wtangle}[2]{\settoheight{\tangle}
 {\epsfxsize=#2 \epsfbox{figures/tangle.#1}}
  \raisebox{-.5\tangle}{\epsfxsize=#2 \epsfbox{figures/tangle.#1}}}

 \usepackage{amsthm}
\newtheorem{thm}{Theorem}
\newtheorem{lemma}[thm]{Lemma}

\newtheorem{prop}[thm]{Proposition}
\newtheorem{cor}[thm]{Corollary}
\newtheorem{defn}[thm]{Definition}

\newtheorem{remark}[thm]{Remark}

\journal{}

\begin{document}

\begin{frontmatter}



\title{Symmetrizers and antisymmetrizers for the BMW algebra}


\author[s]{R. Dipper }
\ead{rdipper@mathematik.uni-stuttgart.de}

\author[c]{J. Hu\corref{cor1}}
\ead{junhu303@yahoo.com.cn}

\author[s]{F. Stoll }
\ead{stoll@mathematik.uni-stuttgart.de}

\address[s]{ Institut f\"ur Algebra und Zahlentheorie,
  Universit\"at Stuttgart,
  Pfaffenwaldring 57, 70569 Stuttgart,
  Germany}
\address[c]{ Department of Mathematics, Beijing Institute of
  Technology, Beijing 100081, PR China}

\cortext[cor1]{Corresponding author}

\begin{abstract}
  Let $n\in\mathds{N}$ and $B_n(r,q)$ be the generic
  Birman-Murakami-Wenzl algebra with respect to indeterminants $r$ and
  $q$. It is known that $B_n(r,q)$ has two distinct linear
  representations generated by two central elements of $B_n(r,q)$
  called the symmetrizer and antisymmetrizer of $B_n(r,q)$. These
  generate for $n\geq 3$ the only one dimensional two sided ideals
  of $B_n(r,q)$ and generalize the corresponding notion for Hecke
  algebras of type $A$. The main result Theorem \ref{theorem:actionE_1} in this paper explicitly 
   determines the coefficients of these
  elements with respect to the graphical basis of  $B_n(r,q)$.
\end{abstract}

\begin{keyword}
 Birman-Murakami-Wenzl algebra \sep symmetrizer \sep antisymmetrizer



\MSC 17B37 \sep 20G05 \sep 20C08

\end{keyword}

\end{frontmatter}



\input epsf

\section{Introduction}

The sum (signed sum) of all group elements of the symmetric group $\Sym_n$
in the group ring $R\Sym_n$ are central elements generating precisely
the one dimensional (two sided) ideals and hence linear characters of
$\Sym_n$ for any commutative ring $R$. They are called symmetrizer
(antisymmetrizer) of $R\Sym_n$ and play an important role in the
representation theory of symmetric groups. This applies in particular
to tensor space, where Schur-Weyl duality  connects representation
theory of symmetric groups and general (or special) linear
groups. All this generalizes to Hecke algebras of type $A$ and the
various quantizations of general (special) linear groups.

The Brauer algebras were defined by Richard Brauer in 1937
\cite{brauer} as centralizing algebras of symplectic and orthogonal
groups acting on tensor space. Quantized versions of those were first
defined and studied independently by Birman and Wenzl
\cite{birmanwenzl} and by Murakami \cite{murakami}. They are multi
parameter algebras which degenerate in the classical limit to Brauer
algebras and are called today Birman-Murakami-Wenzl algebras
(BMW-algebras for short). It is known that there are generalisations
of symmetrizer and antisymmetrizer for BMW-algebras and these
generate the only one dimensional (two sided) ideals of those
provided that $n$ is greater than two (see e.~g.~\cite{huxiao}).

Morton and Wassermann \cite{mortonwassermann} showed that the
BMW-algebra $B_n(r,q)$ with parameters $r$ and $q$ is
isomorphic to Kauffman's tangle algebra \cite{kauffman}, which is
generated by $n$-tangles. Moreover, $B_n(r,q)$ possesses a basis
consisting of tangles whose shadows in the plane (that is not
distinguishing over- and undercrossings) are precisely the Brauer
diagrams producing Brauer's graphical basis of the Brauer algebra.
In this paper we shall determine explicit formulas for the
coefficients of the symmetrizer $x=x_n$ and the antisymmetrizer
$y=y_n$ with repect to such a  basis.

$B_n(r,q)$ comes with a filtration by ideals $I_f$ ($0\leq
f\leq[\frac{n}2]$), where $I_f$ is the ideal generated by tangles with
at least $2f$ horizontal edges. Already in \cite{huxiao} it was shown
that the coefficients of $x$ and $y$ differ on tangles with the same
number of horizontal edges only by a power of $q$ depending on the
number of crossings in the tangle.
However, these coefficients of $x$ and $y$ could not be determined in
\cite{huxiao}.  In this
paper, we shall not only reprove that fact but also explicitly
calculate these coefficients, see Theorem \ref{theorem:actionE_1}. In particular, we verify the conjecture
proposed in  \cite[line -11, page 2921]{huxiao}.

We remark that several other forms of the symmetrizer and antisymmetrizer (and more generally, a complete set of pairwise orthogonal primitive idempotents)
have been given in \cite{heckenbergerschueler}, \cite{IMO}, \cite{LR} by using the Jucys-Murphy operators of BMW algebras and/or fusion procedure and/or inductive constructions.
However, it seems to be difficult to derive the coefficients of
$x_n$ and $y_n$ in the graphical basis of $B_n(r,q)$ from those
formulae. Our approach here for determining the coefficients works
directly with a special basis of $n$-tangles (introduced in
\cite{huxiao}) and calculations with those using the relations for
these basis elements.

\section{Preliminaries}
In this section we recall some basic results on
Birman-Murakami-Wenzl algebras.
\begin{defn}\cite{birmanwenzl,murakami}
  The generic \emph{Birman-Murakami-Wenzl algebra} $B_n=B_n(r,q)$ is
  the unital associative $\mathds{Q}(r,q)$-algebra generated by the
  elements $T_i^{\pm1}$ and $E_i$ for $1\leq i\leq n-1$ subject to the
  relations:
  \begin{alignat}{2}
    T_i-T_i^{-1}&=(q-q^{-1})(1-E_i), &\quad&\text{ for } 1\leq i\leq n-1,\\
    E_i^2&=\delta E_i, &&\text{ for } 1\leq i\leq n-1,\\
    T_iT_{i+1} T_i&=T_{i+1} T_iT_{i+1}, &&\text{ for } 1\leq i\leq n-2,\\
    T_iT_j&=T_jT_i, &&\text{ for }|i-j|>1,\\
    E_iE_{i+1} E_i&=E_i,\;  E_{i+1}E_i E_{i+1}=E_{i+1},
    &&\text{ for } 1\leq i\leq n-2,\\
    T_iT_{i+1}E_i&=E_{i+1} E_i,\; T_{i+1}T_iE_{i+1}=E_iE_{i+1}, &&\text{
      for } 1\leq i\leq n-2,\\
    E_iT_i&=T_iE_i=r^{-1}E_i &&\text{ for } 1\leq i\leq n-1,\\
    E_iT_{i+1}E_i&=rE_i,\;E_{i+1}T_iE_{i+1}=rE_{i+1}, &&\text{ for }
    1\leq i\leq n-2,
      \end{alignat}
  where $\delta=1+\dfrac{r-r^{-1}}{q-q^{-1}}$.
\end{defn}

The Birman-Murakami-Wenzl algebra is isomorphic to an algebra
given in terms of certain diagrams which we will introduce now
\begin{defn}{\cite{kauffman}}
  \begin{enumerate}
  \item A \emph{tangle} is a knot diagram inside a rectangle
    consisting of a finite number of vertices
    in the top and the bottom row of the rectangle
    (not necessarily the same number) and a
    finite number of arcs inside the rectangle such that each vertex
    is connected to another vertex by exactly one arc, arcs either
    connect two vertices or are closed curves. Two tangles are
    \emph{regularly isotopic} if they are related by a sequence of
    Reidemeister Moves II and III (see
    \eqref{eqn:Reidemeister}) and isotopies fixing the
    boundary of
    the rectangle.
    \begin{equation}\label{eqn:Reidemeister}
    \htangle{101}{1cm}= \htangle{102}{1cm}\,\, ,\quad
    \htangle{103}{1cm}= \htangle{104}{1cm}\,\, ,\quad
    \htangle{105}{1cm}= \htangle{106}{1cm} \quad .
    \end{equation}
  \item
    Kauffman's tangle algebra is the $\mathds{Q}(r,q)$-algebra
    generated by all tangles
    with $n$ vertices in the top row and in the bottom row subject to
    the following relations which can be applied to a local disk of
    the tangle:
    \begin{itemize}
      \item Regular isotopy,
      \item $\htangle{107}{.8cm}-\htangle{108}{.8cm}=(q-q^{-1})\left(\,\,
        \htangle{110}{.8cm}-\htangle{109}{.8cm}\,\,\right)$,
      \item 
         $\htangle{111}{1.6cm}=r^{-1}\htangle{113}{1.6cm},\quad
         \htangle{112}{1.6cm}=r\htangle{113}{1.6cm}$,
      \item
        $\htangle{114}{.8cm}=\delta$.
    \end{itemize}
  \end{enumerate}
\end{defn}
The following theorem was shown using Kauffman's invariant of knots.
\begin{thm}{\cite{mortonwassermann}}\label{theorem:isomorphism}
  Kauffman's tangle algebra is isomorphic to the Birman-Murakami-Wenzl
  algebra. The isomorphism maps
  \[
  T_i\mapsto \htangle{116}{.8cm}\;,\quad E_i\mapsto \htangle{115}{.8cm}\;,
  \]
  here the crossing and the horizontal strands respectively
  connect the $i$-th and
  $i+1$-st vertices of the top and bottom row.
\end{thm}

In view of Theorem~\ref{theorem:isomorphism}, each equation in the
Birman-Murakami-Wenzl algebra has
a tangle analogue.
From now on, we will identify  both algebras via the isomorphism
and denote Kauffman's tangle algebra as well by
$B_n=B_n(r,q)$.

A \emph{Brauer $n$-diagram} is a planar graph in a rectangle of a plane, which consist of two rows (top row and bottom row) of $n$ points, with each point joined to precisely one other point (distinct from itself). We label the
points in the top row by $\{1,2,\dots,n\}$ and the points in the bottom row by $\{1^-,2^-,\dots,n^-\}$. Any edge of the form $(i,j)$ or $(i^-,j^-)$, where $1\leq i<j\leq n$, will be called a horizontal edge, while other edges (of the form $(i,j^{-})$) will be called vertical edges. One can obtain from each tangle a Brauer diagram by
forgetting the orientation of crossings. Conversely, for each Brauer
diagram $d$ one can choose a tangle $T(d)$ such that two strands cross
at least once, by choosing orientation of the crossings.
\begin{lemma}[\cite{mortonwassermann}]\label{lemma:tanglebasis}
  For each Brauer $n$-diagram $d$ choose a tangle $T(d)$ by choosing an orientation for
  each crossing (such that two strands never cross twice). Then the
  set
  \[\{T(d)\mid d\text{ Brauer diagram with $n$ vertices on the
    top/bottom row}\}\]
  is a basis for the
  BMW-algebra $B_n$.

\end{lemma}
We fix some notation:
If $f$ is a non-negative number, such that $2f\leq n$, then let
\[
\widehat{E_f}=E_1E_3\cdots E_{2f-1}= \htangle{117}{.8cm}
\]

A \emph{composition}  $\mu\vDash n$ of $n$ is a sequence
$\mu=(\mu_1,\mu_2,\ldots,\mu_k)$ of non-negative integers $\mu_i$
such that $\sum_{i=1}^k\mu_i=n$. If $\mu$ is a composition of $n$ then
let $\Sym_\mu=\Sym_{\mu_1}\times \Sym_{\mu_2}\times \cdots\times
\Sym_{\mu_k}$  be the corresponding
Young subgroup of $\Sym_n$ and let $\DD_\mu$ be the set
of distinguished right coset representatives  of minimal length, such that
$\Sym_n=\Sym_\mu\DD_\mu$.

Throughout, we use the convention that $\Sym_n$ acts on $\{1,2,\dots,n\}$ from the right hand side. In other words, $$
\bigl((a)\sigma)\tau\bigr)=(a)(\sigma\tau),\quad\forall\,a\in\{1,2,\dots,n\},\,\,\forall\,\sigma,\tau\in\Sym_n.
$$
Let  $\HH_f$ be the set of $d\in\Sym_n$ satisfying the following conditions:
\begin{itemize}
\item
  $(1)d<(3)d<\ldots<(2f-1)d$, and
\item$(2i-1)d<(2i)d$ for $i=1,\ldots,f$, and
\item $(2f+1)d<(2f+2)d<\ldots<(n)d$.
\end{itemize}

Note that
$\HH_f$ is in bijection with the possible positions of precisely
$f$ horizontal
edges in the lower part of a Brauer diagram (resp.~in the upper
part). To $d\in\HH_f$ is associated the configuration of $f$
horizontal edges such that the vertex $(2i-1)d$ is connected to the
vertex $(2i)d$ for $i=1,\ldots,f$.
Note that this set is usually denoted by $\mathcal{D}_f$. We
choose $\HH_f$ to avoid confusion with the sets $\DD_\mu$.

If $w\in\Sym_n$ and $w=s_{i_1}s_{i_2}\cdots s_{i_l}$ is a reduced
expression, then let $\ell(w):=l, T_w:=T_{i_1}T_{i_2}\cdots T_{i_l}$. Note that $\ell(w), T_w$ is
independent of the reduced expression of $w$. Then by translating
Lemma~\ref{lemma:tanglebasis} we obtain
\begin{lemma}[\cite{huxiao}]\label{lemma:basis}
  \[\left\{T_uT_\sigma\widehat{E_f}T_d\mid
    0\leq f\leq [\tfrac{n}2],
    u\in\HH_f^{-1},d\in\HH_f,\sigma\in\Sym_{\{2f+1,\ldots,n\}}
  \right\}\]
  is a basis of $B_n$.
\end{lemma}
\section{Symmetrizers and antisymmetrizers}
The Birman-Murakami-Wenzl algebra admits two one-dimensional
representations $\rho_1$ and $\rho_2$ given by
\begin{align*}
  \rho_1(T_i)&=q,&\rho_1(E_i)&=0 &\text{ and }\\
  \rho_2(T_i)&=-q^{-1},&\rho_2(E_i)&=0
\end{align*}
The \emph{symmetrizer} (\emph{antisymmetrizer}) is the quasi-idempotent
generating the ideal in $B_n$ corresponding to the representation
$\rho_1$ ($\rho_2$), that is the unique (up to scalars)
quasi-idempotent $x$ in
$B_n$ such that $xE_i=0$ and $xT_i=qx$ ($xE_i=0$ and
$xT_i=-q^{-1}x$). The goal of this paper will be to determine the
(anti-)symmetrizer.  Since the $\mathbb{Q}$-linear ring automorphism
$B_n\to B_n:T_i\mapsto
T_i, E_i\mapsto E_i, r\mapsto r,q\mapsto -q^{-1}$ maps the
antisymmetrizer to the symmetrizer, we restrict ourselves to determine
the symmetrizer $x\in B_n$.

If $S\subseteq \Sym_n$ is an arbitrary subset of the symmetric group,
let $\widehat{S}=\sum_{w\in S}q^{l(w)}T_w$.
The following lemma gives further information on the symmetrizer $x$:
\begin{lemma}[\cite{huxiao}]\label{lemma:huxiao}
  We have
  \[
  x=\sum_fc_fx_f
  \]
  where $c_f\in\mathbb{Q}(r,q)$ and
  $x_f=\widehat{\HH_f^{-1}}\widehat{\Sym_{\{2f+1,\ldots,n\}}}
    \widehat{E_f}\widehat{\HH_f}$.
\end{lemma}
\begin{remark}\label{remark:multiplicative}
If $S,Y,Z\subseteq\Sym_n$, $S=Y\cdot Z$ and if each element $w\in S$ can be
uniquely written
as $w=yz$ with $y\in Y,z\in Z$ and then $l(w)=l(y)+l(z)$, then
clearly $\widehat{S}=\widehat{Y}\widehat{Z}$.
\end{remark}
 In particular, we have

\begin{lemma}\label{lemma:coset_representatives}
  \begin{enumerate}
  \item
    If $\mu\vDash n$, then
    $\widehat{\Sym_n}=\widehat{\Sym_\mu}\widehat{\DD_\mu}$.
  \item If $\lambda,\mu\vDash n$ such that $\Sym_\lambda\subseteq\Sym_\mu$,
    then
    $\widehat{\DD_\lambda}=\widehat{\Sym_\mu\cap\DD_\lambda}\widehat{\DD_\mu}$.
    Hence, if $\mu=(\mu_1,\ldots,\mu_k)$ and
    $\lambda=(\nu^1,\nu^2,\ldots,\nu^k)$ where $\nu^i$ is
    a composition of $\mu_i$, then we have
    $\widehat{\DD_\lambda}=\widehat{\iota_1(\DD_{\nu^1})}
    \widehat{\iota_2(\DD_{\nu^2})}\cdots
    \widehat{\iota_k(\DD_{\nu^k})}\widehat{\DD_\mu}$
    where $\iota_i$ is the obvious embedding of $\Sym_{\mu_i}$ into $\Sym_n$.

  \end{enumerate}
\end{lemma}

\begin{remark}\label{remark:coset_representatives}
  Applying the antiautomorphism of $B_n$, which maps $T_i$ to $T_i$
  and $E_i$ to $E_i$ to Lemma~\ref{lemma:coset_representatives} shows
  that $\widehat{\Sym_n}=\widehat{\DD_\mu^{-1}}\widehat{\Sym_\mu}$,
  etc.  All diagram identities which will be shown in the next lemmas produce
  new identities by applying this antiautomorphism.
\end{remark}
One may write the results of Lemma~\ref{lemma:coset_representatives}
in terms of diagrams with the following convention:
If a box filled with $\widehat{S}$ occurs in a diagram, we mean the
sum over all $w\in S$, the summands being $q^{l(w)}$ times the diagram
with the box replaced by $T_w$.

For example, the equation  $E_1(1+qT_1)=(1+qr^{-1})E_1$  can be
depicted as
\[
\htangle{88}{2cm}=
\htangle{89}{2cm}
+q  \htangle{90}{2cm}=
(1+qr^{-1})\htangle{89}{2cm}
\]
Note that the upper horizontal edge is not involved in these
calculations, so we can omit it.

\begin{equation}\label{equation:E_1S_2}
  \htangle{1}{1.5cm}=
  (1+qr^{-1})\htangle{2}{1.5cm}
\end{equation}
This equation can be  embedded into a larger piece of tangle, such
that  new tangles arise, the equation holds for the new tangles:
\[
\htangle{91}{3cm}=
(1+qr^{-1})\htangle{92}{3cm}
\]
The assertions of
Lemma~\ref{lemma:coset_representatives} are:

\begin{align}
  \wtangle{4}{3cm}&=
  \wtangle{6}{3cm}\label{equation:Sym}\\
  \wtangle{5}{4cm}&=
  \wtangle{7}{4cm}\label{equation:D_mu}
\end{align}
if $\mu=(\mu_1,\ldots,\mu_k)$,
$\lambda=(\nu^1,\nu^2,\ldots,\nu^k)$ and  $\nu^i$ is
a composition of $\mu_i$.

To determine the coefficients $c_f$, we need several technical
lemmas.

\begin{lemma}\label{lemma:ESym_n}
  Suppose $n\geq 3$. Then
  \begin{equation}\label{equation:ESym_n}
    \wtangle{10}{3cm}=
    \wtangle{11}{3cm}
  \end{equation}
\end{lemma}

\begin{proof}
  For $n=3$ the lemma can be computed directly:
  \begin{eqnarray*}
    \htangle{8}{1.5cm}&=&
    \htangle{12}{1.5cm}
    +q\htangle{13}{1.5cm}
    +q\htangle{14}{1.5cm}
    +q^2\htangle{15}{1.5cm}
    +q^2\htangle{16}{1.5cm}
    +q^3\htangle{17}{1.5cm}\\
    &=&(1+qr^{-1})
    \left(
      q^2\htangle{24}{1cm}
      +q\htangle{26}{1cm}
      +\htangle{25}{1cm}
    \right)
  \end{eqnarray*}
  On the other hand,
  \begin{eqnarray*}
    \htangle{9}{1.5cm}&=&
    \htangle{18}{1.5cm}
    +q\htangle{19}{1.5cm}
    +q\htangle{20}{1.5cm}
    +q^2\htangle{21}{1.5cm}
    +q^2\htangle{22}{1.5cm}
    +q^3\htangle{23}{1.5cm} \\
    &=&(1+qr^{-1})
    \left(
      q^2\htangle{24}{1cm}
      +q\htangle{26}{1cm}
      +\htangle{25}{1cm}
    \right)
  \end{eqnarray*}
  This shows the lemma for $n=3$. For the general case, note that it
  suffices to shift the horizontal edge on top of $\widehat{\Sym_n}$ by
  one to the left (or right). The result follows from
  equation~\eqref{equation:Sym} for $\lambda=(1^k,2,1^{n-k-2})$ and the
  case $n=3$.
\end{proof}

\begin{lemma}\label{lemma:SymDD}
  For $n\geq 2$,we have
  \[
  \widehat{\Sym_n}\cdot\widehat{\DD_{(1,n-1)}}
  =\left(\sum_{i=0}^{n-1}q^{2i}\right)\widehat{\Sym_n}
  +\frac{1-q^2}{q^{-1}r+1}\widehat{\Sym_n}E_1\widehat{\DD_{(2,n-2)}}.
  \]
\end{lemma}
\begin{proof}
  First, note that $\DD_{(1,n-1)}=\{s_1s_2\cdots s_{i-1}\mid
  i=1,\ldots,n\}$ and $\DD_{(2,n-2)}=\{s_2s_3\cdots
  s_{i-1}s_1s_2\cdots s_{j-1}\mid 1\leq j<i\leq n\}$. We first show that
  \begin{equation}\label{equation:SymT1T2cdots}
    \widehat{\Sym_n}T_1T_2\cdots T_{i-1}=q^{i-1} \widehat{\Sym_n}
    +\frac{1-q^2}{r+q}
    \widehat{\Sym_n}E_1\sum_{j=1}^{i-1}q^{j-1}T_2T_3\cdots
    T_{i-1}T_1T_2\cdots T_{j-1}
  \end{equation}
  for $i=1,\ldots,n$.
  We argue by induction on $i$.
  Equation~\eqref{equation:SymT1T2cdots} is trivial for
  $i=1$. Suppose, equation~\eqref{equation:SymT1T2cdots} holds for
  $i$.
  We have
  \begin{align*}
    \widehat{\Sym_n}&T_1T_2\cdots
    T_{i}=(\widehat{\Sym_n}T_1T_2\cdots T_{i-1})T_{i}\\
    &=
    q^{i-1} \widehat{\Sym_n}T_{i}
    +\frac{1-q^2}{r+q}
    \widehat{\Sym_n}E_1\sum_{j=1}^{i-1}q^{j-1}T_2T_3\cdots
    T_{i-1}T_1T_2\cdots T_{j-1}T_{i}
  \end{align*}
  Note that $T_{i}$ commutes with $T_1T_2\cdots T_{j-1}$, thus the
  second part of the right hand side is exactly the sum over
  $j=1,\ldots,i-1$ in equation~\eqref{equation:SymT1T2cdots} for $i+1$
  instead of $i$.

To simplify the notations, we shall often use $a^k$ to represents the $k$-tuple $\underbrace{a,a,\dots,a}_{\text{$k$ copies}}$.
  Remark~\ref{remark:coset_representatives} for $\lambda=(1^{i-1},2,1^{n-i-1})$
  shows that $ \widehat{\Sym_n}= \widehat{\DD_\lambda^{-1}}(1+qT_i)$.
  It can be verified by direct calculation that
 \begin{eqnarray*}
    (1+qT_i)T_i
    &=&(1+qT_i)\cdot\left(q+\frac{1-q^2}{r+q}E_i\right).
  \end{eqnarray*}
 Since $T_xT_y=T_{xy}$ whenever $\ell(xy)=\ell(x)+\ell(y)$, it is easy to see that $(1+qT_i)$ is a right factor of $\widehat{\Sym_n}$. Therefore,
we obtain
  \begin{align*}
    q^{i-1} \widehat{\Sym_n}T_{i}&=   q^{i-1}
    \widehat{\Sym_n}\left(q+\frac{1-q^2}{r+q}E_i\right)
    =q^i\widehat{\Sym_n}+\frac{1-q^2}{r+q} q^{i-1}
    \widehat{\Sym_n}E_i
  \end{align*}
  $ q^i\widehat{\Sym_n}$ is the summand $q^{i-1}\widehat{\Sym_n}$ in
  equation~\eqref{equation:SymT1T2cdots} with $i$ replaced by $i+1$.
  Apply Remark~\ref{remark:coset_representatives} to
  equation~\eqref{equation:ESym_n} to obtain
\begin{align*}
  \widehat{\Sym_n}E_i&=\wtangle{28}{3cm}
  =\wtangle{29}{3cm}=\wtangle{30}{3cm}\\
  &= \widehat{\Sym_n}E_1T_2T_3\cdots T_{i}T_1T_2\cdots T_{i-1}
 \end{align*}
 Thus, $\frac{1-q^2}{r+q} q^{i-1}\widehat{\Sym_n}E_i$ is the summand
 for $j=i$ in
  equation~\eqref{equation:SymT1T2cdots} with $i$ replaced by $i+1$.
  This shows equation~\eqref{equation:SymT1T2cdots}. 
  The result follows by summing $q^{i-1}$ times the term in
  equation~\eqref{equation:SymT1T2cdots}, since
  \[\sum_{1\leq j<i\leq n}q^{j-1+i-1}T_2T_3\cdots
  T_{i-1}T_1T_2\cdots T_{j-1}=q\widehat{\DD_{(2,n-2)}}\]

\end{proof}
The diagram version of Lemma~\ref{lemma:SymDD} is
\begin{equation}\label{equation:SymDD}
  \wtangle{31}{3cm}
  =\left(\sum_{i=0}^{n-1}q^{2i}\right)
  \wtangle{4}{3cm}
  +\frac{1-q^2}{q^{-1}r+1}
  \wtangle{32}{3cm}
\end{equation}

Next, we will show several identities involving $\HH_f$. Note, that
$\HH_f=\HH_f(n)$ not only depends on $f$, but also on $n$. If
$j-i\geq 2f$, we write
$\HH_f^{\{i+1,\ldots,j\}}$ for the image of  $\HH_f(j-i)$ under the
embedding $T_k\mapsto T_{i+k}$, $E_k\mapsto E_{i+k}$.
We will use a similar notation for $\DD_\mu$.

If
$\widehat{\HH_f}$ occurs in a diagram, then $i$ and $j$ might be obtained
from the diagram: $i$ is the number of vertical edges on the left of
the box filled with $\widehat{\HH_f}$ and $j-i$ is the number of edges
going into the box. Thus, in the diagrams we simply write $\widehat{\HH_f(n)}$
or simply $\widehat{\HH_f}$ if $n$ is clear from the context.


\begin{lemma}\label{lemma:H_fn}
  We have
  \[
  \widehat{\HH_{f}(n)}=
  \widehat{\HH_{f}(2f)}
  \widehat{\DD_{(2f,n-2f)}}
  \]
In diagrams,
\begin{equation}\label{equation:H_fn}
  \wtangle{33}{4cm}
  =\wtangle{34}{4cm}
\end{equation}
\end{lemma}
\begin{proof}
  Since $\DD_{(2f,n-2f)}$ is a set of distinguished right coset
  representatives of $\Sym_{(2f,n-2f)}$ in $\Sym_n$ and $\HH_{f}(2f)\subseteq
  \Sym_{(2f,n-2f)}$,
  each element of
  $\HH_{f}(2f)\DD_{(2f,n-2f)}$ can be
  uniquely written as a product of elements of these sets, and lengths
  add. Thus
  $ \widehat{\HH_{f}(2f)}
  \widehat{\DD_{(2f,n-2f)}}=
  \widehat{\HH_{f}(2f)\DD_{(2f,n-2f)}}$
  and it suffices to show that
  $\HH_{f}(2f)\DD_{(2f,n-2f)}=\HH_{f}(n)$.
  Both sets have the same cardinality,
  namely $\frac{n!}{2^ff!(n-2f)!}$
  and it is easy to see that
  $\HH_{f}(2f)\DD_{(2f,n-2f)}\subseteq\HH_{f}(n)$
  by verifying the defining conditions of $\HH_f$.
\end{proof}
Lemma~\ref{lemma:H_fn} shows that $\HH_f(n)$ can be deduced from
$\HH_f(2f)$. The following lemmas deal with the case $n=2f$.
\begin{lemma}\label{lemma:H_f(2f)}
  We have
  \[
  \widehat{\HH_{f}(2f)}
  =\widehat{\HH_{f-1}^{\{3,\ldots,2f\}}}\widehat{\DD_{(1,2f-2)}^{\{2,\ldots,2f\}}}
  \]
  or in diagrams
  \begin{equation}
    \wtangle{35}{4cm}=
    \wtangle{36}{4cm}
    \label{equation:H_f-1}
  \end{equation}
\end{lemma}
\begin{proof} This can be proved with the same arguments used for
  Lemma~\ref{lemma:H_fn} by
  showing that
  $\HH_{f-1}^{\{3,\ldots,2f\}}\DD_{(1,2f-2)}^{\{2,\ldots,2f\}}=
  \HH_{f}^{\{1,\ldots,2f\}}$. Note that both sets have cardinality
  $\frac{(2f)!}{2^f\cdot f!}$.
\end{proof}

\begin{lemma}
  We have
  \[
  \widehat{\mathcal{H}_2^{\{1,2,3,4\}}}
  \widehat{\mathcal{D}_{(3,2f-4)}^{\{2,\ldots,2f\}}}=
  \widehat{\mathcal{D}_{(1,2,2f-4)}^{\{2,\ldots,2f\}}}
  \]
  or
  \begin{equation}\label{equation:H_21234}
    \wtangle{38}{4cm}
    =\wtangle{37}{4cm}
  \end{equation}
\end{lemma}
\begin{proof}
  By equation~\eqref{equation:H_f-1} for $f=2$ (note that
  $\widehat{H_1(1)})=1$) and
  equation~\eqref{equation:D_mu}, both sides are equal to
  \[
  \wtangle{39}{4cm}
  \]
\end{proof}

Recall that we consider the case $n=2f$.
Let $\double{s}_i=s_{2i}s_{2i-1}s_{2i+1}s_{2i}=(2i-1,2i+1)(2i,2i+2)$.
The subgroup of $\Sym_{2f}$ generated by $\double{s}_i$ for
$i=1,\ldots,f-1$ is isomorphic to the symmetric group on $f$ letters.
We will
denote this subgroup by $\double{S}_f$.

Let  $\double{w}\in \double{S}_f$. Then $\double{w}\HH_f$ is the set of $d\in
\DD_{(2^f)}$ such that
$(1)\double{w}^{-1}d<(3)\double{w}^{-1}d<(5)\double{w}^{-1}d <\ldots
<(2f-1)\double{w}^{-1}d$. Thus we have
\[\DD_{(2^f)}=\double{S}_f\HH_f\]
and each element $d\in
\DD_{(2^f)}$ can be uniquely written as $d=\double{w}d'$ with
$\double{w}\in\double{S}_f ,d'\in\HH_f$. In general,
$l(d)=l(\double{w})+l(d')$ is not true.

Let
$\double{S}_{(1,f-1)}
=\left\langle\double{s}_2,\double{s}_3,\ldots,\double{s}_{f-1}
\right\rangle$, which is a Young subgroup in $\double{S}_f$, and let $
\double{D}_{(1,f-1)}=\{\double{s}_1\double{s}_2\cdots\double{s}_{k}\mid
k=0,\ldots,f-1\}$ be the corresponding set of coset
representatives. Then we have

\begin{lemma}
  \[
  \widehat{\double{D}_{(1,f-1)}\HH_f(2f)}=\widehat{\HH_{f-1}^{\{3,\ldots,2f\}}}
  \widehat{\DD_{(2,2f-2)}}
  \]
  In diagrams,
  \begin{equation}\label{equation:doubleDH_f}
    \wtangle{40}{4cm}
    =\wtangle{41}{4cm}
  \end{equation}
\end{lemma}
\begin{proof}
  Clearly, $\widehat{\HH_{f-1}^{\{3,\ldots,2f\}}}
  \widehat{\DD_{(2,2f-2)}}=
  \widehat{\HH_{f-1}^{\{3,\ldots,2f\}}\DD_{(2,2f-2)}}$. Thus it
  suffices to show that
  $\HH_{f-1}^{\{3,\ldots,2f\}}\DD_{(2,2f-2)}=\double{D}_{(1,f-1)}\HH_f$. Both
  sets have cardinality
  $\frac{(2f)!}{2^f(f-1)!}$. Now,  $\double{D}_{(1,f-1)}\HH_f$ is the set of
  $d\in\DD_{(2^f)}$, such that  $(3)d<(5)d<\ldots<(2f-1)d$. On the
  other hand, $\HH_{f-1}^{\{3,\ldots,2f\}}\DD_{(2,2f-2)}$ is a subset
  of $\DD_{(2^f)}$ and each element of
  $\HH_{f-1}^{\{3,\ldots,2f\}}\DD_{(2,2f-2)}$  satisfies the condition
  $(3)d<(5)d<\ldots<(2f-1)d$.
\end{proof}

\begin{lemma}
  We have
  \[
  \widehat{\double{s}_1\double{D}_{(1,f-2)}^{\{2,\ldots,f\}}\HH_f(2f)}=
  \widehat{\double{s}_1\HH_2^{\{1,\ldots,4\}}}
  \widehat{\HH_{f-2}^{\{5,\ldots,2f\}}}\widehat{\DD_{(3,2f-4)}^{\{2,\ldots,2f\}}}
  \]

  In diagrams,
  \begin{equation}\label{equation:s_1doubleDH_f}
    \wtangle{42}{4cm}
    =  \wtangle{43}{4cm}
  \end{equation}
\end{lemma}

\begin{proof}
  Note that
  $\double{D}_{(1,f-1)}
  =\{1\}\dot\cup\double{s}_1\double{D}_{(1,f-2)}^{\{2,\ldots,f\}}$. Thus,
  $\double{s}_1\double{D}_{(1,f-2)}^{\{2,\ldots,f\}}\HH_f(2f)$ is the
  set of $d\in\DD_{(2^f)}$, such that  $(3)d<(5)d<\ldots<(2f-1)d$ and
  $(1)d>(3)d$. The elements of $\double{s}_1\HH_2^{\{1,\ldots,4\}}
  \HH_{f-2}^{\{5,\ldots,n\}}\DD_{(3,2f-4)}^{\{2,\ldots,n\}}$ satisfy
  this condition and both sets have cardinality
  $\frac{(2f)!(f-1)}{2^ff!}$.
\end{proof}

 \begin{prop}\label{prop:e_fH_f}
   Let $f\geq 0$ and $\HH_f=\HH_f(n)$. Let $\widehat{e_f}$ be the
   bottom part of  $\widehat{E_f}$. Then we have
   \begin{enumerate}
   \item
   \[ \widehat{e_f}\;\widehat{\DD_{(1^{2f},n-2f)}}
   =(1+qr^{-1})^f \widehat{e_f}\;\widehat{\DD_{(2^{f},n-2f)}}\]
 \item \[\widehat{e_f}\;\widehat{\DD_{(2^{f},n-2f)}}
   = a_f\widehat{e_f}\;\widehat{\HH_f}
   \]
   where $a_0=1$ and $a_i=a_{i-1}\cdot\sum_{k=0}^{i-1}q^{2k}$ for
   $i>0$.
 \end{enumerate}

\end{prop}
\begin{proof}
  For $f=0$, the proposition is trivial.
  \begin{enumerate}
  \item
    Note that $\DD_{(1,1)}=\Sym_2$. We have
    \begin{align*}
      &\wtangle{44}{4cm}
    \overset{\eqref{equation:D_mu}}
    =\wtangle{46}{4cm}\\
    &\overset{\eqref{equation:E_1S_2}}=(1+qr^{-1})^f
    \wtangle{45}{4cm}
    \end{align*}
    which shows the first claim.
  \item
    We have
    \begin{align*}
      \wtangle{45}{4cm}&
      \overset{\eqref{equation:D_mu}}=
      \wtangle{48}{4cm}\text{
        and }\\\bigskip\\
      \wtangle{47}{4cm}&
      \overset{\eqref{equation:H_fn}}=
      \wtangle{49}{4cm}
    \end{align*}
    Thus, without loss of generality we can assume that $n=2f$.
    We show that
    \begin{align}
      \widehat{e_f}\;\widehat{\double{D}_{(1,f-1)}\HH_f}&
      =\left(\sum_{k=0}^{f-1}q^{2k}\right)
      \widehat{e_f}\;\widehat{\HH_f}\text{ and}\label{equation:proof1}\\
        \widehat{e_f}\;\widehat{\DD_{(2^f)}}&=a_f\widehat{e_f}\;\widehat{\HH_f}
        \label{equation:proof2}
    \end{align}
    For $f=1$ these equations are trivial. Let $f=2$, then
    $\double{D}_{(1,1)}=\double{S}_2=\{1,\double{s}_1\}$ and thus
    $\double{D}_{(1,1)}\HH_2=\DD_{(2,2)}$. We have
    \begin{align*}
      \{T_d\mid d\in\HH_2\}&=\left\{\quad
        \wtangle{50}{1cm}\quad,\quad
        \wtangle{51}{1cm}\quad,\quad
        \wtangle{52}{1cm}\quad
      \right\}\\
      \{T_d\mid d\in\double{s}_1\HH_2\}&=\left\{\quad
        \wtangle{53}{1cm}\quad,\quad
        \wtangle{54}{1cm}\quad,\quad
        \wtangle{55}{1cm}\quad
      \right\}
    \end{align*}
    Then
    \begin{align}
      \widehat{e_2}\widehat{\double{s}_1\HH_2}&=
      q^2\wtangle{57}{1cm}
      +q^3\wtangle{58}{1cm}
      +q^4\wtangle{56}{1cm}\notag\\
      \smallskip\notag\\
      &=q^2\left(
        q^2\wtangle{57}{1cm}
        +q\wtangle{59}{1cm}
        +\wtangle{56}{1cm}
      \right)=
      q^2  \widehat{e_2}\widehat{\HH_2}\label{equation:case2}
      \end{align}
      It follows that equations~\eqref{equation:proof1} and
      \eqref{equation:proof2} hold for $f=2$. Suppose now that
      \eqref{equation:proof1} and
      \eqref{equation:proof2} hold for $f-1$. Then,
      $\double{D}_{(1,f-1)}=\{1\}\dot\cup
      \double{s}_1\double{D}_{(1,f-2)}^{\{2,\ldots,f\}}$ and thus
      \begin{align*}
        \widehat{e_f}\;\widehat{\double{D}_{(1,f-1)}\HH_f}
        &=\wtangle{62}{4cm}\\
        &\overset{\eqref{equation:s_1doubleDH_f}}=
        \wtangle{61}{4cm}
        +\wtangle{64}{4cm}
      \end{align*}
      By \eqref{equation:case2}, this equals
      \begin{align*}
        &
        \wtangle{61}{4cm}
        +q^2\wtangle{63}{4cm}\\
        &\overset{\eqref{equation:H_21234}}=
        \wtangle{61}{4cm}
        +q^2\wtangle{65}{4cm}\\
      \end{align*}
      Applying equation~\eqref{equation:D_mu} to the compositions
      $\lambda=(1,2,2f-4)$ and $\mu=(1,2f-2)$  and then
      equation~\eqref{equation:doubleDH_f} for $f-1$, we obtain
     \begin{align*}
      \widehat{e_f}\;\widehat{\double{D}_{(1,f-1)}\HH_f}&
     =\wtangle{61}{4cm}
        +q^2\wtangle{66}{4cm}\\
        &\overset{\text{induction}}=
        \wtangle{61}{4cm}
        +q^2\sum_{k=0}^{f-2}q^{2k}
        \wtangle{67}{4cm}\\
        &\overset{\eqref{equation:H_f-1}}= \sum_{k=0}^{f-1}q^{2k}
        \wtangle{61}{4cm}
      \end{align*}
      Thus, equation~\eqref{equation:proof1} holds for $f$. Now,
      \begin{align*}
        &\widehat{e_f}\;\widehat{\DD_{(2^f)}}\\
        &=
        \wtangle{93}{4cm}
        \overset{\eqref{equation:D_mu}}=
        \wtangle{95}{4cm}
        \overset{\text{ind.}}=a_{f-1}
        \wtangle{94}{4cm}
      \end{align*}
      Then
      \eqref{equation:proof2} follows from \eqref{equation:doubleDH_f}
      and \eqref{equation:proof1}.
  \end{enumerate}

\end{proof}

Recall that
$x_f=\widehat{\HH_f^{-1}}\widehat{\Sym_{\{2f+1,\ldots,n\}}}
\widehat{E_f}\widehat{H_f}$ and let $y_f=\widehat{\Sym_n}
\widehat{E_f}\widehat{H_f}$. We have
\begin{align*}
  y_f&=\widehat{\Sym_n}\widehat{E_f}\widehat{H_f}
  \overset{\eqref{equation:Sym}}=
  \widehat{\DD_{(1^{2f},n-2f)}^{-1}}\widehat{\Sym_{(1^{2f},n-2f)}}
  \widehat{E_f}\widehat{H_f}\\
  &=\widehat{\DD_{(1^{2f},n-2f)}^{-1}}
  \widehat{E_f}\widehat{\Sym_{(1^{2f},n-2f)}}\widehat{H_f}\\
  &\overset{\text{Prop.}\ref{prop:e_fH_f}}=(1+qr^{-1})^fa_f
  \widehat{\HH_f^{-1}}
  \widehat{E_f}\widehat{\Sym_{(1^{2f},n-2f)}}\widehat{H_f}=(1+qr^{-1})^fa_fx_f
\end{align*}

In order to compute $y_fE_1$ and thus $x_fE_1$, we decompose $\HH_f$
into $5$ pairwise disjoint subsets. Note that if $d\in\HH_f$ then
$(1)d^{-1}=1$ or $(1)d^{-1}=2f+1$. Let
\begin{eqnarray*}
  \HH_f^{(a)}&=&\{d\in\HH_f\mid (1)d^{-1}=1,(2)d^{-1}=2\}\\
  \HH_f^{(b)}&=&\{d\in\HH_f\mid (1)d^{-1}=1,(2)d^{-1}=3\}\\
  \HH_f^{(c)}&=&\{d\in\HH_f\mid (1)d^{-1}=1,(2)d^{-1}=2f+1\}\\
  \HH_f^{(d)}&=&\{d\in\HH_f\mid (1)d^{-1}=2f+1,(2)d^{-1}=1\}\\
  \HH_f^{(e)}&=&\{d\in\HH_f\mid (1)d^{-1}=2f+1,(2)d^{-1}=2f+2\}
\end{eqnarray*}

We have
\begin{lemma}\label{lemma:H_f^a} $\HH_f$ is a disjoint union of $\HH_f^{(a)}$, $\HH_f^{(b)}$, $\HH_f^{(c)}$, $\HH_f^{(d)}$ and $\HH_f^{(e)}$, and
  \begin{eqnarray*}
  \HH_f^{(a)}&=&\HH_{f-1}^{\{3,\ldots,n\}}\\
  \HH_f^{(b)}&=&s_2\HH_{f-2}^{\{5,\ldots,n\}}\DD_{(1,1,n-4)}^{\{3,\ldots,n\}}\\
  \HH_f^{(c)}&=&s_{2f}s_{2f-1}\cdots s_2\HH_{f-1}^{\{4,\ldots,n\}}
  \DD_{(1,n-3)}^{\{3,\ldots,n\}}\\
  \HH_f^{(d)}&=&s_{2f}s_{2f-1}\cdots s_2s_1\HH_{f-1}^{\{4,\ldots,n\}}
  \DD_{(1,n-3)}^{\{3,\ldots,n\}}\\
  \HH_f^{(e)}&=&s_{2f}s_{2f-1}\cdots s_1 s_{2f+1}s_{2f}\cdots s_2
  \HH_{f}^{\{3,\ldots,n\}}
\end{eqnarray*}
Furthermore, in each of these products lengths add.
\end{lemma}
\begin{proof} It follows directly from the definition that $\HH_f$ is a disjoint union of
$\HH_f^{(a)}$, $\HH_f^{(b)}$, $\HH_f^{(c)}$, $\HH_f^{(d)}$ and $\HH_f^{(e)}$. For the products on the right hand side of above five equalities, it is easy to check that lengths add. It is also easy to
see that the right hand sides are subsets of the left hand sides respectively.
By a simple counting, the number of elements of the right hand sides add up
to the cardinality of $\HH_f$. This implies the above five equalities.
\end{proof}

\begin{prop}
  We have
  \begin{align*}
    y_0E_1&=\widehat{\Sym_n}E_1\\
    y_fE_1&=\frac{(rq^{2n-2f-2}-q^{-1})(1+qr^{-1})}{q-q^{-1}}\widehat{\Sym_n}
    \widehat{E_f}\widehat{\HH_{f-1}^{\{3,\ldots,n\}}}
    +q^{2f}\widehat{\Sym_n}\widehat{E_{f+1}}
    \widehat{\HH_{f}^{\{3,\ldots,n\}}}
  \end{align*}
  for $1\leq f\leq[\frac{n}2]$.
  Here, we set $\widehat{\Sym_n}\widehat{E_{f+1}}
  \widehat{\HH_{f}^{\{3,\ldots,n\}}}=0$, if $f=[\frac{n}2]$ i.~e.~if
  $n=2f$ or $n=2f+1$.
\end{prop}
\begin{proof}
  The first equality is trivial.
  Let $f\geq 1$. In order to obtain
  $y_fE_1=\widehat{\Sym_n}\widehat{E_f}\widehat{\HH_{f}}E_1$, we compute
  $\widehat{\Sym_n}\widehat{E_f}\widehat{\HH_{f}^{(\alpha)}}E_1$
  for each $\alpha\in\{a,b,c,d,e\}$.

  First, let $\alpha=a$. By Lemma~\ref{lemma:H_f^a},
  $\HH_f^{(a)}=\HH_{f-1}^{\{3,\ldots,n\}}$ and thus
  \begin{equation}
    \widehat{\Sym_n}\widehat{E_f}\widehat{\HH_{f}^{(a)}}E_1=
    \widehat{\Sym_n}\widehat{E_f}\widehat{\HH_{f-1}^{\{3,\ldots,n\}}}E_1=
    \widehat{\Sym_n}\widehat{E_f}E_1\widehat{\HH_{f-1}^{\{3,\ldots,n\}}}
    =
    \delta\widehat{\Sym_n}\widehat{E_f}\widehat{\HH_{f-1}^{\{3,\ldots,n\}}}
    \label{equation:H_f^a}
  \end{equation}
  Note, that  \eqref{equation:H_f^a} holds for all
  $f\in\{1,\ldots,[\frac{n}2]\}$.
  If $\alpha=b$ and $f\geq 2$, then by Lemma~\ref{lemma:H_f^a}
  $\widehat{\HH_f^{(b)}}
  =qT_2\widehat{\HH_{f-2}^{\{5,\ldots,n\}}}
  \widehat{\DD_{(1,1,n-4)}^{\{3,\ldots,n\}}}$ and thus
  \begin{align*}
    \widehat{\Sym_n}\widehat{E_f}\widehat{\HH_{f}^{(b)}}E_1&=
    q\wtangle{68}{4cm}=
    qr\wtangle{69}{4cm}
  \end{align*}
  By Proposition~\ref{prop:e_fH_f}, this equals
  \begin{align}
    &\frac{qr(1+qr^{-1})^{2-f}}{a_{f-2}}
    \wtangle{70}{4cm}
    \overset{\eqref{equation:D_mu}}=
    \frac{qr(1+qr^{-1})^{2-f}}{a_{f-2}}
    \wtangle{71}{4cm}\notag
    \\
    &=
    qr(1+qr^{-1})\frac{a_{f-1}}{a_{f-2}}
    \wtangle{72}{4cm}\notag\\
    &=qr(1+qr^{-1})\left(\sum_{k=0}^{f-2}q^{2k}\right)  \widehat{\Sym_n}
    \widehat{E_f}\widehat{\HH_{f-1}^{\{3,\ldots,n\}}}  \label{equation:H_f^b}
  \end{align}
  Note that $\HH_{f}^{(b)}=\emptyset$ and $\sum_{k=0}^{f-2}q^{2k}=0$
  if $f=1$. Hence, this equation is also valid for  $f=1$.

  Consider now the cases $\alpha=c$ and $\alpha=d$. First, let
  $f\neq[\frac{n}2]$. Then we have
  \begin{align*}
    &\widehat{\Sym_n}\widehat{E_f}
    \left(\widehat{\HH_{f}^{(c)}}+\widehat{\HH_{f}^{(d)}}\right)E_1\\
    &=
    \widehat{\Sym_n}\widehat{E_f}
    \left(q^{2f-1}T_{2f}T_{2f-1}\cdots T_2+
      q^{2f}T_{2f}T_{2f-1}\cdots T_2T_1\right)\widehat{\HH_{f-1}^{\{4,\ldots,n\}}}
    \widehat{\DD_{(1,n-3)}^{\{3,\ldots,n\}}}
    E_1\\
    &=q^{2f-1}\wtangle{73}{4cm}+
    q^{2f}\wtangle{74}{4cm}\\
    &=q^{2f}(rq^{-1}+1)
    \wtangle{75}{4cm}
    =\frac{q^{2f}(rq^{-1}+1)}{a_{f-1}(1+qr^{-1})^{f-1}}
    \wtangle{76}{4cm}
  \end{align*}
  By equation~\eqref{equation:D_mu}, we have
  \[
  \widehat{\DD_{(1^{2f-2},n-2f-1)}^{\{4,\ldots,n\}}}
  \widehat{\DD_{(1,n-3)}^{\{3,\ldots,n\}}}=
  \widehat{\DD_{(1^{2f-1},n-2f-1)}^{\{3,\ldots,n\}}}
  =\widehat{\Sym_{2f-1}^{\{3,\ldots,2f+1\}}}
  \widehat{\DD_{(2f-1,n-2f-1)}^{\{3,\ldots,n\}}}
  \]
  Applying equation~\eqref{equation:ESym_n}, we obtain
  \begin{align*}&\widehat{\Sym_n}\widehat{E_f}
    \left(\widehat{\HH_{f}^{(c)}}+\widehat{\HH_{f}^{(d)}}\right)E_1\\
    &=\frac{q^{2f}(rq^{-1}+1)}{a_{f-1}(1+qr^{-1})^{f-1}}
    \wtangle{77}{4cm}
    =\frac{q^{2f}(rq^{-1}+1)}{a_{f-1}(1+qr^{-1})^{f-1}}
    \wtangle{78}{4cm}
  \end{align*}

  By equation~\eqref{equation:D_mu} and \eqref{equation:Sym},
  \begin{align*}
    \widehat{\Sym_{2f-1}^{\{3,\ldots,2f+1\}}}
    \widehat{\DD_{(2f-1,n-2f-1)}^{\{3,\ldots,n\}}}
    &= \widehat{\DD_{(1^{2f-1},n-2f-1)}^{\{3,\ldots,n\}}}
    =\widehat{\DD_{(1,n-2f-1)}^{\{2f+1,\ldots,n\}}}
    \widehat{\DD_{(1^{2f-2},n-2f)}^{\{3,\ldots,n\}}}\\
    \widehat{\Sym_n}&=
    \widehat{\DD_{(1^{2f},n-2f)}^{-1}}\widehat{\Sym_{n-2f}^{\{2f+1,\ldots,n\}}}
  \end{align*}
  Thus by equation~\eqref{equation:SymDD}, we have
  \begin{align*}
    &\wtangle{78}{4cm}
    =\wtangle{79}{4cm}\\
    &= \sum_{k=0}^{n-2f-1}q^{2k}
    \wtangle{81}{4cm}
    +\frac{1-q^2}{q^{-1}r+1}\cdot
    \wtangle{80}{4cm}
 \end{align*}

 The first summand equals
 \[
 \left(\sum_{k=0}^{n-2f-1}q^{2k}\right)
 \widehat{\Sym_n}\widehat{E_f}
 \widehat{\DD_{(1^{2f-2},n-2f)}}=
 a_{f-1}(1+qr^{-1})^{f-1}  \left(\sum_{k=0}^{n-2f-1}q^{2k}\right)
 \widehat{\Sym_n}\widehat{E_f}
 \widehat{\HH_{f-1}^{\{3,\ldots,n\}}}
 \]
 By Proposition~\ref{prop:e_fH_f}, the second summand is
 \begin{align*}
   &\frac{1-q^2}{q^{-1}r+1} \widehat{\Sym_n}\widehat{E_{f+1}}
   \widehat{\DD_{(2,n-2f-2)}^{\{2f+1,\ldots,n\}}}
   \widehat{\DD_{(1^{2f-2},n-2f)}^{\{3,\ldots,n\}}}\\
   &=\frac{(1-q^2)(1+qr^{-1})^{f-1}}{q^{-1}r+1} \widehat{\Sym_n}\widehat{E_{f+1}}
   \widehat{\DD_{(2,n-2f-2)}^{\{2f+1,\ldots,n\}}}
   \widehat{\DD_{(2^{f-1},n-2f)}^{\{3,\ldots,n\}}}\\
   &=\frac{(1-q^2)(1+qr^{-1})^{f-1}}{q^{-1}r+1} \widehat{\Sym_n}\widehat{E_{f+1}}
   \widehat{\DD_{(2^f,n-2f-2)}^{\{3,\ldots,n\}}}\\
   &=\frac{(1-q^2)(1+qr^{-1})^{f-1}a_f}{q^{-1}r+1}
   \widehat{\Sym_n}\widehat{E_{f+1}}
   \widehat{\HH_f^{\{3,\ldots,n\}}}
 \end{align*}
 Taking this altogether, we get
 \begin{align}
   &\widehat{\Sym_n}\widehat{E_f}
   \left(\widehat{\HH_{f}^{(c)}}+\widehat{\HH_{f}^{(d)}}\right)E_1\notag\\
   &= \frac{q^{2f}(rq^{-1}+1)}{a_{f-1}(1+qr^{-1})^{f-1}}\cdot
   \left( a_{f-1}(1+qr^{-1})^{f-1}  \left(\sum_{k=0}^{n-2f-1}q^{2k}\right)
     \widehat{\Sym_n}\widehat{E_f}
     \widehat{\HH_{f-1}^{\{3,\ldots,n\}}}\right.\notag\\
     &\left.+\frac{(1-q^2)(1+qr^{-1})^{f-1}a_f}{q^{-1}r+1}
     \widehat{\Sym_n}\widehat{E_{f+1}}
     \widehat{\HH_f^{\{3,\ldots,n\}}} \right)\notag\\
   &= q^{2f}(rq^{-1}+1)
   \left(\sum_{k=0}^{n-2f-1}q^{2k}\right)
   \widehat{\Sym_n}\widehat{E_f}
   \widehat{\HH_{f-1}^{\{3,\ldots,n\}}}+q^{2f}(1-q^2)
   \left(\sum_{k=0}^{f-1}q^{2k}\right)
   \widehat{\Sym_n}\widehat{E_{f+1}}
   \widehat{\HH_f^{\{3,\ldots,n\}}} \label{equation:H_f^cd}
   \end{align}
   If $f=[\frac{n}2]$, i.~e.~$n=2f$ or $n=2f+1$, then the same
   considerations can be done, except for omitting the second
   summand. But $\widehat{\Sym_n}\widehat{E_{f+1}}
   \widehat{\HH_f^{\{3,\ldots,n\}}}=0$ by our convention. Thus
   equation~\eqref{equation:H_f^cd} holds for
   $f\in\{1,\ldots,[\frac{n}2]\}$.

   Finally, if $\alpha=e$ and $f\neq[\frac{n}2]$, then
   \begin{align}
   \widehat{\Sym_n}\widehat{E_f}
   \widehat{\HH_{f}^{(e)}}E_1
   &= q^{4f}\widehat{\Sym_n}\widehat{E_f}
   T_{2f}T_{2f-1}\cdots T_1 T_{2f+1}T_{2f}\cdots T_2
   \widehat{\HH_{f}^{\{3,\ldots,n\}}}E_1\notag\\
   &=q^{4f}\wtangle{82}{4cm}
   =q^{4f}\wtangle{83}{4cm}
   \notag\\
   &= q^{4f}\widehat{\Sym_n}\widehat{E_{f+1}}
   \widehat{\HH_f^{\{3,\ldots,n\}}} \label{equation:H_f^e}
   \end{align}
   Again, if $f=[\frac{n}2]$ then this equation also holds.
   Adding \eqref{equation:H_f^a},
   \eqref{equation:H_f^b}, \eqref{equation:H_f^cd} and
   \eqref{equation:H_f^e} yields
the proposition.

\end{proof}

\begin{thm}\label{theorem:actionE_1}
  We have
  \[
  \frac{c_f}{c_{f-1}}
  =-\frac{(q-q^{-1})q^{2f-2}\sum_{k=0}^{f-1}q^{2k}}{rq^{2n-2f-2}-q^{-1}}
  =-\frac{q^{4f-3}-q^{2f-3}}{rq^{2n-2f-2}-q^{-1}}
  \]
  for $f\geq 1$.
\end{thm}
\begin{proof}
  Let $b_0=0$ and
  $b_f=\frac{(1+qr^{-1})(rq^{2n-2f-2}-q^{-1})}{q-q^{-1}}$ for $f\geq 1$,
  $d_f=q^{2f}$ and $\EE_f= \widehat{\Sym_n}\widehat{E_{f+1}}
  \widehat{\HH_f^{\{3,\ldots,n\}}}$.  Note that in view of
  Proposition~\ref{prop:e_fH_f},
  \[\EE_f=
  a_{f+1}(1+qr^{-1})^{f+1}\widehat{\HH_{f+1}^{-1}}
  \widehat{\Sym_{\{2f+1,\ldots,n\}}}
  \widehat{E_{f+1}}
  \widehat{\HH_f^{\{3,\ldots,n\}}},\]
  and this is a linear combination
  of basis elements. In particular, the set $\{\EE_f,\mid 0\leq f\leq
  [\frac{n}2]\}$ is linearly  independent.
  We have
  \begin{align*}
    0=\sum_{0\leq f\leq [n/2]}c_fx_fE_1&=
    \frac{1}{a_f(1+qr^{-1})^f}\sum_{0\leq f\leq [n/2]}c_fy_fE_1\\
    &=
    \sum_{0\leq f\leq [n/2]}\left(\frac{b_fc_f}{a_f(1+qr^{-1})^f}\EE_{f-1}
    +\frac{d_fc_f}{a_f(1+qr^{-1})^f}\EE_f\right)
  \end{align*}
  It follows that  for $f\geq 1$
  \[
  \frac{b_fc_f}{a_f(1+qr^{-1})^f}
  +\frac{d_{f-1}c_{f-1}}{a_{f-1}(1+qr^{-1})^{f-1}}=0
  \]
  and thus
  \begin{align*}
    \frac{c_f}{c_{f-1}}
    &=-\frac{d_{f-1}a_f(1+qr^{-1})^{f}}{a_{f-1}(1+qr^{-1})^{f-1}b_f}
    =-\frac{d_{f-1}(1+qr^{-1})\sum_{k=0}^{f-1}q^{2k}}{b_f}\\
    &=-\frac{q^{2f-2}(1+qr^{-1})(q-q^{-1})\sum_{k=0}^{f-1}q^{2k}}
    {(1+qr^{-1})(rq^{2n-2f-2}-q^{-1})}
    =-\frac{q^{2f-2}(q-q^{-1})\sum_{k=0}^{f-1}q^{2k}}
    {(rq^{2n-2f-2}-q^{-1})}
  \end{align*}
\end{proof}
\begin{cor}\label{cor:actionE_iT_i}
  Let $x=\sum_{f=0}^{[\frac{n}2]}c_fx_f$ be as in
  Theorem~\ref{theorem:actionE_1}. Then we have
  \begin{align*}
    xE_i&=0\\
    xT_i&=qx
  \end{align*}
for $i=1,\ldots,n-1$.
\end{cor}
\begin{proof}
  First note that $x_f=\frac{1}{a_f(1+qr^{-1})^f}\widehat{\HH_f^{-1}}
  \widehat{E_f}\widehat{\Sym_n}$ by the same arguments that showed
  $x_f=\frac{1}{a_f(1+qr^{-1})^f}y_f$. Thus $x=y
  \widehat{E_f}\widehat{\Sym_n}$ for some $y\in B_n$. In the proof of
  Lemma~\ref{lemma:SymDD} we showed that $\widehat{\Sym_n}T_i=
  q\widehat{\Sym_n}+\frac{1-q^2}{r+q}\widehat{\Sym_n}E_i$. Thus it
  suffices to show that $xE_i=0$ for $i=1,\ldots,n-1$.

  For $i=1$ this
  holds by definition of $x$. Note that in the proof of
  Theorem~\ref{theorem:actionE_1}, the bottom horizontal edge of $E_1$
  was not involved in our calculations. Thus we actually have
  \[
   \wtangle{84}{3cm}=0
  \]
  But then by Lemma~\ref{lemma:ESym_n},
  \begin{align*}
    xE_i&=
    \wtangle{85}{3cm}
    = \wtangle{86}{3cm}
    = \wtangle{87}{3cm}=0
  \end{align*}

\end{proof}

\begin{remark}
  Note that Corollary~\ref{cor:actionE_iT_i} provides
  an alternative proof of
  Lemma~\ref{lemma:huxiao}.
\end{remark}
In particular, as an application we can now
prove  a conjecture of \cite{huxiao}.
For $l\in\mathds{N}$ let
$B_{n,q}:=B_n(-q^{2l+1},q)$ be the specialized BMW-algebra over
$\mathds{Q}(q)$ and let
$V$ be a $2l$-dimensional vector space over $\mathds{Q}(q)$. Then
$B_{n,q}$ acts on
tensor space $V^{\otimes n}$ and satisfies Schur-Weyl
duality together with the action of the symplectic quantum group
$U_q(\mathfrak{sp}_{2l})$ on
tensor space. In \cite{huxiao}, a generator of the annihilator
 in $B_{n,q}$ of $V^{\otimes n}$ for $n>l$ is constructed. It is given
 as the specialized  antisymmetrizer
$y=y_{l+1}$ of $B_{l+1,q}$, embedded into $B_{n,q}$.
The conjecture proposed in \cite[Page 2921, line
-11]{huxiao} now states:
\begin{cor}
  The generator $y=y_{l+1}$ of the annihilator can be rescaled, such that the
  coefficient of the basis element $T_uT_\sigma\hat{E_f}T_d$  (see
  Lemma~\ref{lemma:basis}) in $y$ is $q^k\cdot(-q)^{-l(u)-l(\sigma)-l(d)}$ for
  some $k\in\mathds{Z}$.
  In particular, it is up to sign a power of $q$.
\end{cor}
\begin{proof}
  Let $c_f(r,q)$ be the coefficients $c_f$ from
  Lemma~\ref{lemma:huxiao}. Then the
  antisymmetrizer has the form
  $$y=\sum_{0\leq f\leq [n/2]}c_f(r,-q^{-1})y_f$$
  with $y_f=\sum_{u,\sigma,d} (-q)^{-l(u)-l(\sigma)-l(d)}T_uT_\sigma\hat{E_f}T_d$.
  We have
  $$
  \frac{c_f(r,-q^{-1})}{c_{f-1}(r,-q^{-1})}
  =\frac{q^{3-4f}(1-q^{2f})}{q+rq^{2+2f-2n}}.
  $$
  Now  specializing $r$ to $-q^{2l+1}$ and $n$ to $l+1$ (as in the
  setting of \cite[Theorem 5.4]{huxiao}),
  we get $q^{2-4f}$, which is a power of $q$.
  This
  proves the conjecture proposed in \cite{huxiao}.
\end{proof}
\bigskip

\centerline{Acknowledgement}
\medskip

The second author was supported, in part, by the National Natural Science Foundation of China. He also thanks the Alexander von Humboldt Foundation for supporting a working visit in Stuttgart in 2011.

\end{document}